\documentclass[11pt,reqno]{article}
\usepackage{graphicx}
\usepackage{latexsym,epsfig,mathrsfs,amsmath,amssymb}
\usepackage{amstext}
\usepackage{hyperref}
\usepackage{verbatim}
\usepackage{algorithm}
\usepackage{algpseudocode}
\usepackage{color}
\newtheorem{theorem}{Theorem}

\numberwithin{example}{section}
\numberwithin{equation}{section}
\numberwithin{theorem}{section}

\usepackage{url}
\usepackage{listings}
\usepackage{xcolor}
\lstdefinestyle{mystyle}{%
	backgroundcolor=\color{gray!12}, commentstyle=\color{codegreen},
	keywordstyle=\color{magenta},
	basicstyle=\ttfamily\footnotesize,
	breakatwhitespace=false,
	breaklines=true,
	captionpos=b,
	keepspaces=true,
	numbers=left,
	numbersep=5pt,
	showspaces=false,
	showstringspaces=false,
	showtabs=false,
	tabsize=2  
}
\allowdisplaybreaks

\begin{document}

\begin{flushleft}
  {\bf\large On Average Modulus of Random Polynomials Over a Unit Circle and Disc}
\end{flushleft}
\parindent=0mm \vspace{.3in}

{\bf{Sajad A. Sheikh$^{a,\star}$ \quad and \quad Mohammad Ibrahim Mir$^{a}$}}

\parindent=0mm \vspace{.1in}
{\small \it
$^{a}$Department of  Mathematics, University of Kashmir, South Campus, Anantnag 192101, Jammu and Kashmir, India.
E-mail: $\text{sajadsheikh@uok.edu.in};\text{ibrahimmath80@gmail.com}$}

\parindent=0mm \vspace{.1in}
{\small {\bf Abstract.}
This article presents some interesting and novel results concerning the average modulus of random polynomials on the unit circle and the unit disc, with coefficients distributed as standard normal variates. The paper also introduces new results concerning the bounds of the maximum modulus of random polynomials with coefficients distributed as independently as Gaussian and uniform variates, utilizing probability principles to derive findings about the likelihood of the maximum modulus exceeding a specific threshold, using Markov inequality as the primary probabilistic tool. These findings and the approach can potentially initiate the study of a rich class of problems concerning the norms of random polynomials.

\parindent=0mm \vspace{.1in}

\parindent=0mm \vspace{.1in}

\parindent=0mm \vspace{.2in}
\section{Introduction and Motivation}

The study of polynomials and their properties has always been a central theme in algebra and complex analysis given their pure interest and pervasive applications. Polynomials, being one of the most basic yet effective mathematical entities, find widespread applications across a variety of disciplines, including physics, engineering, computer science, and economics \cite{rah,borp,mil}.They, along with their random counterparts called random polynomials, serve as mathematical tools in diverse problems, such as the design and analysis of algorithms, the modeling of physical phenomena, the development of computational methods for solving differential equations, and the understanding of the behavior of complex systems \cite{borp,Bharuch}. Investigating their behavior and  the methods to quantify it constitutes a significant research endeavor in mathematical sciences. This paper delves into this rich and fertile domain, focusing on the probabilistic study of the modulus of random polynomials over a unit circle, specifically those with coefficients distributed independently and identically as a standard normal variate.

 \parindent=0mm \vspace{.1in}
 
 In our previous works,the methodologies laid out in \cite{sheikh2023} and \cite{sheikh2024} introduce a novel probabilistic approach for the study of the distribution of zeros and critical points. In \cite{sheikh2023}, we embarked on an exploration of the Eneström–Kakeya theorem within a probabilistic framework, revealing insights into the stability of random polynomials without the constraints of coefficient ordering, thereby broadening the theorem's applicability across various probability distributions of coefficients. Building on this, \cite{sheikh2024} delved into the intricate relationship between the distribution of zeros and critical points of random polynomials with some distributions on coefficients or zeros , employing probabilistic inequalities to establish bounds on their moduli. These investigations set the stage for further exploration into the dynamic interplay of polynomial coefficients and their roots in a probabilistic setting. The current study extends these methodologies to a more focused investigation of the norms of random polynomials on the unit circle, aiming to obtain some new probabilistic results and insights into their behavior and distribution.
In the literature, various norms of complex polynomials have been defined and studied. One of the most extensively studied norms is the maximum modulus over the unit circle, also called the infinity norm, which plays a crucial role in many classical inequalities about polynomials like Bernstein inequality and Markov inequality. The maximum norm over a unit circle serves as a powerful lens through which the behavior of the polynomial can be scrutinized. This relationship is also deeply intertwined with the study of mean values of a polynomial and the theory of growth estimates. The maximum modulus of a polynomial on the unit circle,in particular, is a critical measure of the polynomial's growth  \cite{rah}. Following \cite{rah},  this relationship can be formally be  described as follows:\\
 For a function $F$ continuous on $|z|=1$, if we define the quantities \\
 \begin{equation}
     \mathcal M (F;r)=\left( \frac{1}{2 \pi} \int_{-\pi}^{\pi} |F(r e^{\iota t}|^p dt\right)^{1/p}
 \end{equation} 
 and 
 \begin{equation}
    \lim_{p \to \infty} \mathcal M (F;r)= \max_{\substack{ -\pi  \leq t \leq \pi}} |F(r e^{\iota t})|=\mathcal{M}_{\infty }(F;r), 
 \end{equation}
then the quantity $$ \frac{ \mathcal M(F;1)}{\mathcal{M}_{\infty }(F;r)}$$ is  a measure of the growth of $F$ on the unit circle. Equally significantly, we have the following result \cite{rah}:\\
 If $P$ is  a polynomial of degree at most $n$, then
\begin{equation}
M_{\infty}(P ; 1)=\max_{\substack{|z|=1}}|P(z)| \leq \sqrt{n+1}\left(\frac{1}{2 \pi} \int_{-\pi}^\pi \mid P(r e^{i t}) \mid ^2
dt \right)^{1 / 2}.
\end{equation}
The famous classical Bernstein Inequality \cite{borp} relates the maximum norm of a polynomial to that of its derivative in the form of the following inequality :
\parindent=00mm
Let  $P(z)$  be a complex polynomial of degree $n$  .  Then,
\begin{equation}    
\max_{|z|=1} |P'(z)| \leq n \max_{|z|=1} |P(z)|.
\end{equation}
Regarding the mean of the squared modulus of a polynomial $ P(z)=\sum_0^n c_i z^k $ over the unit circle,it can be easily shown that the average of $ |P(z)|^2$  over $|z|=1$ is given by   $\sum_0^n {c_i}^2.$ It is interesting to note that this well-known fact can also be established by  Viewing $|P(z)|$ as a function of a uniform random variate over $[0,2 \pi].$ Hence  for $z=e^{\iota t},$ we can write :
\begin{equation}
\mathbb{E}(|P(z|^2) =\sum_0^n {c_i}^2
\end{equation}  
 Since for a  random variable $\mathbb E(X) \leq \sqrt{\mathbb E(X^2)}$, it immediately follows:
\begin{equation}
\mathbb{E}(|P(z|) \leq  \sqrt{\sum_0^n {c_i}^2},
\end{equation}
where $|z|=1$
\subsection{Motivation} Polynomials and their random counterparts, random polynomials crop up across a plethora of scientific fields. In signal processing and control systems, complex polynomials often represent transfer functions or system responses \cite{oppen99, kuo92}. The modulus of a polynomial (i.e., the magnitude of its coefficients) provides insights into the system's stability and frequency response. For example, in designing a low-pass filter, understanding the modulus of the transfer function helps ensure that the system attenuates high-frequency noise while preserving the desired signal 
Researchers use polynomial-based models to estimate volatility, hedge portfolios, and simulate market scenarios \cite{rebon, glasser}. Complex polynomials appear in statistical physics models \cite{land80}. Analyzing their zeros informs phase transitions and critical phenomena. Despite the extensive study of polynomial norms in deterministic settings \cite {mil,deewan,neder,frap,govil,Aziz} , the probabilistic study of norms of random polynomials has received comparatively less attention. This gap is significant given that random polynomials model real-world phenomena where randomness is inherent, such as in signal analysis ,statistical mechanics, financial modeling, and complex network analysis. Probabilistic methods allow for a more nuanced understanding of the behavior of polynomials under random influences, leading to results that can inform both theory and applications.

In the literature, however, estimates for various systems of functions $\{f_j\}_{j=1}^{n}$ and random variables $\{\xi_j\}_{j=1}^{n}$ have been widely applied in analysis since the 1930s. In 1954, Salem and Zygmund \cite{salzyg} established several estimates for the uniform norm of random trigonometric polynomials. In particular, they showed that
\begin{equation}
    \mathbb{E} \left\| \sum_{k=-n}^{n} r_k(\omega) e^{i k t} \right\|_{\infty} \asymp (n \log n)^{1/2},
\end{equation}
where $r_k(\omega)$ are the Rademacher functions. Here and further, the expression $A_n \asymp B_n$ stands for $cA_n \leq B_n \leq CA_n$ with some constants $c$ and  $C$. 
Notably, Grigoriev \cite{grigoriev2002} established estimates for the integral-uniform norm and the uniform norm of random polynomials. His work demonstrated the utility of integral-uniform norms in analyzing function distributions and provided lower estimates for these norms, highlighting their importance in various function systems. 

Building on their work and methods,Grigoriev \cite{grigoriev2002} established  some estimates for the mathematical expectation of norms of random polynomials of the type
\begin{equation}
    \sum_{j=1}^{n} a_j \xi_j(\omega) f_j(x),
\end{equation}
where $\{\xi_j\}_{j=1}^{n}$ is a set of independent random variables defined on $(\Omega, P)$, and $\{f_j\}_{j=1}^{n}$ is a set of functions on another probability space $(X, \mu)$. 
Another note-worthy work in this regard is Borwein and Lockhart \cite{borwein2001} which studies the expected $L^p$ norms of random polynomials on the boundary of the unit disc. They provided precise estimates for the growth rates of these norms, extending the understanding of the behavior of random trigonometric polynomials.  
These foundational results highlight the potential for further exploration and the need for new methodologies to tackle unresolved questions about the norms of random polynomials. Given the importance of expected values of random quantities in the study of random systems, we undertake to initiate the study of this and related problems. Specifically, we study the expected value of the modulus of a random polynomial on the unit circle and the unit disc with coefficients distributed as standard normal variaties. The choice of the standard normal distribution is motivated by its prevalence in statistical modeling and its ability to capture a wide range of natural phenomena. Additionally, we present a result regarding the bounds of the maximum modulus and utilize the principles of probability to derive results concerning the likelihood of the maximum modulus surpassing a specific threshold. This analysis is conducted by employing the Markov inequality \cite{Ross,Feller} as a probabilistic tool.  Our paper aims to take a different direction the methods and tools from the theory of probability and obtain some new results regarding probabilistic bounds of the moduli of certain polynomials.  Specifically, we focus on random polynomials with coefficients that adhere to a standard normal distribution.

The exploration is poised to reveal a probabilistic viewpoint on norms of polynomials on the unit circle. The insights gleaned from this study can contribute to a deeper understanding of the behavior of random polynomials under these conditions. The norm of a polynomial plays a fundamental role in the study of certain properties like flatness introduced by Littlewood in the context of some special random polynomials. Furthermore, our methods can be readily adapted to investigate similar questions about random polynomials with different coefficient distributions.

\parindent=0mm \vspace{.1in}

The rest of the paper is organized as follows: Section 2 covers some preliminaries and lemmas, including the general modulus formula of a polynomial, an integral related to advanced calculus, and the probability density function (pdf) of the sum of two folded normal variates. The paper then transitions to the "Main Results" in Section 3, which includes several key theorems. Theorem 3.1 establishes properties of random polynomials with normally distributed coefficients, providing the expected squared magnitude and an upper bound for the expected magnitude on the unit circle. Theorem 3.2 presents a probabilistic bound for the modulus of these polynomials surpassing a specific threshold on the unit circle. Theorem 3.3 extends the analysis to the unit disc, offering results on the average modulus over the unit disc, including a theorem and proof regarding the expected value of the modulus. Further results in this section include probabilistic bounds on the maximum modulus of random polynomials with uniformly distributed coefficients and bounds on the maximum modulus for polynomials with standard normal coefficients. The concluding sections discuss the implications of these findings and potential directions for future research.

---

\section{Preliminaries and Lemmas}

 \subsection{Lemma 1: The General Modulus Formula of a Polynomial}\label{sec:general_modulus_formula}
   For  a polynomial $ P(z)=\displaystyle\sum_0^n a_k z_k $ with  real  coefficients, the square of the  modulus over the unit circle $ |z| =1$ at $ z=e^{\iota t},0\leq t \leq 2 \pi$, is given by: 
\begin{equation}\label{2.1}
|P(z)|^2 =\sum_{j,k=0}^{n} a_j {a_k}  \cos(j-k)t.
\end{equation}
{\bf Proof:}
\parindent=0mm \vspace{.1in}
Consider a polynomial $ P(z)=\displaystyle\sum_0^n a_k z_k $ with  complex coefficients, over the unit circle $ |z| =1$, we have
\begin{equation}\label{2.2}
 z= e^{ \iota t }, 0  \leq t  \leq 2 \pi. \end{equation}
Therefore, 
\begin{align}
    |P(z)|^2 &=Re\left(P(z).\overline{P(z)} \right) \\     
    &=Re\left(\sum_0^n a_j z^j .  \sum_0^n \bar{a_k} \bar{z}^k \right)\notag\\
    &= Re\left(\sum_{j,k=0}^{n}a_j \bar{a_k} z^j \bar{z}^k\right).
\end{align}
 For $z$ on the unit circle, we have upon  using  equation \eqref{2.1} that
  \begin{equation}\label{2.4}
   |P(z)|^2 = \operatorname{Re}\left( \sum_{j,k=0}^{n} e^{-\iota (j-k)t} \right).
  \end{equation}
When the coefficients are real, then the above equation \eqref{2.4} reduces to the following equation 
 \begin{equation}\label{2.5}
 |P(z)|^2 =\sum_{j,k=0}^{n} a_j {a_k}  \cos(j-k)t.
 \end{equation}
 
\subsection{Lemma 2} The following integral can be found in any standard book on advanced calculus or advanced engineering  mathematics text books. For more details and derivations  we refer the reader  to \cite{jain-iyengar}:
\begin{equation}
\displaystyle \int_{-\infty}^{\infty} x^m e^{-ax^2}dx=\dfrac{\sqrt\pi}{a^{m+1/2}}\dfrac{m!}{2^m}\binom{m}{2}!
\end{equation}
}
\subsection{Lemma 3}  The pdf of the sum $X$ of two folded normal variates is given by \cite{Ross,Feller}: 
\begin{equation}
   \displaystyle \dfrac{2 e^{-\frac{x^2}{4}} erf \left(\dfrac{x}{2}\right)}{\sqrt{\pi}}.
\end{equation}

\section{Main Results }
\subsection{Expected value of the Modulus on the unit circle}
In the following theorem, we find the average of the squared modulus of the ensemble of random polynomials of degree $n$ with iid and normally distributed coefficients.

\parindent=0mm \vspace{.1in}
\begin{theorem}
Let $\mathcal{P}_n$ denotes the ensemble of all polynomials of degree $n$ with random coefficients, and if $d\mu_P$ and $d\mu_z$ denote the uniform  probability measures on $\mathcal{P}_n$ and the unit circle.Further ,let  $P(z)$ random polynomial  of the form
\begin{equation} 
P(z) = \sum_{j=0}^{n} A_jz^j,
\end{equation}
where each coefficients $A_j's$ are random variable independently and identically distributed ($iid$) standard normal distribution, i.e., $A_j \sim N(0,1)$ for all $j = 0, 1, 2, \ldots, n$. For this polynomial, when evaluated on the unit circle ($|z|=1$), the following properties hold within the space of all such polynomials $\mathcal{P}_n$:
\begin{itemize}
  \item The expected squared magnitude of $P(z)$ is given by
  \begin{equation}
  \mathbb{E}_{|z|=1, P \in \mathcal{P}_n}[|P(z)|^2] = n+1,
  \end{equation}
  \item The expected magnitude of $P(z)$ is bounded above by
  \begin{equation}
  \mathbb{E}_{|z|=1, P \in \mathcal{P}_n}[|P(z)|] \leq \sqrt{n+1}.
  \end{equation}
\end{itemize}
\end{theorem}

    {\bf Proof:} Let $\mathcal{P}_n$ denotes the ensemble of all polynomials of degree $n$ with random coefficients, and if $d\mu_P$ and $d\mu_z$ denote the uniform  probability measures on $\mathcal{P}_n$ and the unit circle, then we have:
    \begin{equation}
    \mathbb {E}_{|w|=1,P \in \mathcal{P}_n }[|P(z)|^2]=\int_{\mathcal{P}_n} \left( \int_{|z|=1} |P(e^{i\theta})|^2 \, d\mu_z \right) \, d\mu_P.
    \end{equation}
     To simplify the above, we observe  that  for any polynomial
\begin{equation*}
             p(z)=\sum_{0}^{n} a_j z^j ,  
           \end{equation*}
the expected value of $|P(z)|^2$ over $|z|=1$ can easily be seen to be  $\displaystyle \sum_{0}^{n} a_j^2.$
By using equation (2.5), we see the inner integral in equation $(3.4)$  average $|P(e^{\iota t})|^2$ over the circle is nothing but the average of the  function  
$\displaystyle\sum_{j,k=0}^{n} a_j {a_k}  \cos(j-k)t ,$ where $t \in (0,2 \pi)$  is viewed as a uniform random variable in the interval  $[0,2 \pi] $. Noting that the pdf of $t$ is $f(t)=\frac{1}{2 \pi} $, we see that the expectation is $\displaystyle\sum_0^n A_j^2.$ 
So the average of all random polynomials of degree $n$ over the unit circle is  given by :
 \begin{equation}
   \mathbb {E}_{|w|=1,P \in \mathcal{P}_n }[|P(z)|^2]= \int_{\mathcal{P}_n}  |P(e^{i\theta})|^2 \,  \, d\mu_P.
    \end{equation}
    Now  ,
\begin{equation*}
 A_j\sim N(0,1) , \, j=0,1,2\cdots,n, 
\end{equation*} ,
therefore $A_j^2 \sim \chi_{1}^{2},$ so that $$\mathbb{E}(A_j^2)=1,j=0,1,2 \cdots,n$$
Alternatively we can calculate it as follows: 
\begin{equation*}
\mathbb E(A_j^2)=\frac{1}{\sqrt2\pi}\int_{-\infty}^{\infty} A_j^2e^{-A_j^2/2}dA_j.
\end{equation*}
Letting $A_j=u,$ we have the expectation 
\begin{equation}
\mathbb E (A_j^2)=\frac{1}{\sqrt2\pi}\int_{-\infty}^{\infty} u^2e^{-u^2/2}dk.
\end{equation}
Using Lemma 2, with $m=2 ,a=1/2,$ above equation simplifies to:  

\begin{equation*}
\int_{-\infty}^{\infty} k^2.e^{-k^2/2}dx=\frac{1}{(1/2)^3/2}.\frac{2!}{4.1}\sqrt{\pi}=(\sqrt2)^3.\sqrt\pi/2=1.
\end{equation*}
Hence, we have for each $j=0,1,2,\cdots,n$
\begin{equation}
  \mathbb E (A_j^2)=1. 
\end{equation}
Thus, in view of the independence of $A_j's$, we see that the average 
 $|P(z)|^2$ for all random polynomial $P(z)$ on $|z|=1$ with coefficients distributed as standard normal variaties  
\begin{align*}
  \mathbb E[|P(z)|^2]_{|z|=1,
  P \in \mathcal{P}_n }&=n+1\\
  \implies \mathbb E(\sum_{0}^{n} A_j^2) &=n+1.
\end{align*}
Now, using the fact that for a random variable $X,$
\begin{equation*}
\mathbb E(X) \leq \sqrt{\mathbb E (X^2)},
\end{equation*}
we immediately see that:
\begin{equation}
   \mathbb E_{|z|=1,P \in \mathcal{P}_n }[|P(z)|] \leq \sqrt{n+1}.
\end{equation}

\parindent=8mm \vspace{.1in}
As an immediate consequence of the above theorem, we have the following theorem:
\begin{theorem}
Let $P(z) = \displaystyle\sum_{i=0}^n A_i z^i$ be a random polynomial where the coefficients $a_i$ are independent and follow the standard normal distribution $N(0,1)$. Then for any $c > 0$, on the unit circle, the probability that the modulus of $P(z)$ on $|z|=1$ for a uniformly chosen $z$ is larger than $c\sqrt{n+1}$ is less than or equal to $\dfrac{1}{c^2}$.
\end{theorem}

\textbf{Proof:} The expected value of the modulus square of the polynomial on the unit circle is $$\mathbb{E}[|P(z)|^2]_{|z|=1} =  n+1.$$

Therefore, using  the Markov's inequality for any $c > 0$, we have

\begin{align}
\mathbb{P}(|P(z)| & > c\sqrt{n+1})\\
& \mathbb{P}(|P(z)|^2  > c^2(n+1)\\
& \leq \frac{\mathbb{E}[|P(z)|^2]}{c^2(n+1)}\\ 
&= \frac{1}{c^2}.
\end{align}

Therefore, the probability that the modulus of $P(z)$ is larger than $c\sqrt{n+1}$ is less than or equal to $\dfrac{1}{c^2}$.

 \subsection{ Average over the unit disc} 
 
 \begin{theorem}
 Let $\mathcal{P}_n$ denotes the ensemble of all polynomials of degree $n$ with random coefficients, and if $d\mu_P$ and $d\mu_z$ denote the uniform  probability measures on $\mathcal{P}_n$ and the unit circle, and  let $P(z)$,be  a random polynomial of the form:
\begin{equation} 
P(z)=\sum_{0}^{n} A_jz^j
\end{equation}
with $A_j \sim N(0,1),\, j=0,1,2,\cdots,n$. , let $H_n$ be the $nth$ partial sum of the harmonic series. Then, we have:\\
{\bf (a)}
  \begin{equation}
  \mathbb {E}_{|z|\leq 1,P \in \mathcal{P}_n }[|P(z)|^2]=\left(H_{2n+1}-\frac{1}{2}H_n\right) \approx \sqrt{log(2\sqrt{n})},
  \end{equation}    
  for large  $n.$\\
  {\bf (b)}\begin{equation}
  \mathbb {E}_{|z|\leq 1,P \in \mathcal{P}_n}[|P(z)|] \leq  \sqrt{\left(H_{2n+1}-\frac{1}{2}H_n \right)}.
  \end{equation}
 \end{theorem}
 {\bf Proof:} Similar to our notation in the previous theorem, let $\mathcal{P}_n$ denotes the ensemble of all polynomials of degree $n$ with random coefficients, and if $d\mu'_P$ and $d\mu_z$ denote the uniform  probability measures on $\mathcal{P}_n$ and the unit disc, then we have:
    \begin{equation}
    \mathbb {E}_{|w|=1,P \in \mathcal{P}_n }[|P(z)|^2]=\int_{\mathcal{P}_n} \left( \int_{|z|=1} |P(e^{i\theta})|^2 \, d\mu'_z \right) \, d\mu'_P.
    \end{equation}
 To facilitate the evaluation of the above integral, we first of all  show with the help of Lemma 1 that the  average squared modulus  of any polynomial $p(z)|=\sum_0^n a_i z^i$ over $|z|\leq1|$ is given by
  \begin{equation*}
  \mathbb E[|p(z)|^2]_{|z| \leq 1}=\sum \frac{a_j^2}{2j+1}.
  \end{equation*}

Consider a polynomial $p(z)=\sum_0^n a_k z^k$ with real coefficients. The value of $|p(z)|^2$ over $|z|=r \quad$ is given by
$$
|p(z)|^2=\sum_{j, k=0}^n a_j a_k r^{j+k} \cos (j-k) t.
$$
Hence, the average of $|P(z)|^2$ over the ensemble of $\mathcal{P}_n$  over  the unit circle is given by
\begin{align}
\mathbb E \left[|p(z)|^2\right]=\frac{1}{2 \pi} \int_0^{2 \pi} \int_0^1 \sum_{j, k=0}^n a_j a_k r^{j+k} \cos (j-k) t \, dr\, dt =\sum_{j=0}^{n} \frac{a_j^2}{2j+1}
\end{align}
 In the light of the  equation $(3.17),$ the inner integral in equation $ (3.16)$ simplifies to 
 
 \begin{equation}
     \displaystyle\int_{|z|=1} |P(e^{i\theta})| \, d\mu'_z.
  \end{equation}
 
 For the sake of compactness denoting by $B$ the integral  $ \displaystyle\int_{\mathcal{P}_n} \left( \int_{|z|=1} |P(e^{i\theta})| \, d\mu'_z \right) \, d\mu'_P,$ it follows that the  average of the modulus over the unit disc  of the ensemble of all random polynomials of degree $n$ with coefficients distributed as standard normal variances is given by:
  \begin{equation*}
     B= \mathbb E \left(\frac{A_0^2}{1}+\frac{A_1^2}{3}+\cdots+\frac{A_n^2}{2n+1}\right).
  \end{equation*}
  In view of the independence of $A_i's,$ and the fact that $\mathbb E (A_j^2)=1$, we have:
  \begin{equation*}
  A=\left(1+\frac{1}{3}+\cdots+\frac{1}{2n+1}\right).
  \end{equation*}
  A simple algebraic reveals that the above can also be written as:
  \begin{equation*}
  B=\left(H_{2n+1}-\dfrac{1}{2}H_n \right)
  \end{equation*}
  where $H_n=\sum\dfrac{1}{n}.$
  
  For large $n,$ using asymptotic properties of $H_n,$
  \begin{equation*}
  B \sim log\left( \frac{2n+1}{\sqrt{n}}\right) \sim log\left(2  \sqrt{n}\right).
  \end{equation*}
  That establishes part (a) of the theorem. Part (b) of the theorem is an immediate consequence of the fact that for any random variable $X,$ the expectation never exceeds the square root of the expectation of its square. 

\subsection{Probabilistic Bounds on \texorpdfstring{$\displaystyle\max_{|z|=1}P(z)$}{max P(z) for |z|=1} for uniform distribution on the coefficients}

\begin{theorem}
  Let $\mathcal{P}_n$ denotes the ensemble of all polynomials of degree $n$ with  standard random coefficients, and if $d\mu_P$  denote the uniform  probability measures on $\mathcal{P}_n.$ Let $ P(z)=\displaystyle\sum_{i=0}^n A_i z^i$  be a random polynomial with independently and identically and  uniformly distributed coefficients in $[-1,1]$ and $Y=\displaystyle\max_{|z|=1}P(z).$ Then we have the following:
   \begin{equation}
  P\left(\max_{|z|=1} |P(z)| \geq c\right) \leq \dfrac{1}{c}.
  \end{equation}
\end{theorem}
{\bf Proof:} In order to prove the above theorem, we use the result that for a polynomial $ p(z)=\displaystyle\sum_0^n a_i z^i,$ we have
 \begin{equation}
     \max_{|z|=1} \geq |a_0|+|a_n|.
 \end{equation}
 The proof of which can be obtained using \cite{maxbound}. However, following \cite{mathst}  it can be proved more elegantly as follows.\\
 Denoting  by $\omega $ an $nth $ root of unity, we have:
\begin{equation}
\sum_{j=0}^{n-1} {(\omega^k)}^j = \begin{cases} 
n, & k \in \{ 0, n \} \\ 
0, & 1 \le k \le n-1.
\end{cases}
  \end{equation}
Hence
\begin{align}
\frac1n \sum_{j=0}^{n-1} f(\omega^j z) 
&= \frac1n \sum_{j=0}^{n-1} \sum_{k=0}^{n} a_k {(\omega^j)}^k z^k \\
&= \frac1n \sum_{k=0}^{n} a_k z^k \sum_{j=0}^{n-1} {(\omega^k)}^j = a_0 + a_nz^n.
\end{align}

Consequently, we get
\begin{align}
|a_0| + |a_n| &= \max_{|z|=1} |a_0 + a_nz^n | \\
&=  \max_{|z|=1} \Big| \frac1n \sum_{j=0}^{n-1} p(\omega^j z) \Big| \\
&\le    \frac1n \sum_{j=0}^{n-1} \max_{|z|=1} \big| f(\omega^j z) \big|
= \max_{|z|=1} |p(z)|.
\end{align}
We now use the above proposition to establish theorem. Firstly, for the random polynomial $P(z)=\displaystyle\sum_{i=0}^n A_i z^i $, let $X = |A_0| + |A_n|$, where $A_i \sim \mathbb U(-1,1)$ are uniformly and independently distributed. 
To compute $\mathbb E[X]$, we note that $|A_0|$ and $|A_n|$ are independent and identically distributed, and each has mean $\dfrac{1}{2}$ and variance $\dfrac{1}{3}$. Therefore, we have:
\begin{align*}
\mathbb E[X] &= \mathbb E[|A_0| + |A_n|] \\
&= 2\mathbb E[|A_0|] \\
&= 2\int_0^1 x \,dx \\
&= 1.
\end{align*}
Therefore, we have:
$$
P(X \geq c) \leq \dfrac{1}{c}.
$$
Next, we have $ \displaystyle\max_{|z|=1} |P(z)| \geq |A_0| + |A_n|$, so we can apply the same argument as before to obtain:
\begin{align*}
P\left(\max_{|z|=1} |P(z)| \leq c\right) \geq  \dfrac{\mathbb E\left( \max_{|z|=1} |P(z)|\right) }{c}   \geq \dfrac{1}{c}. 
\end{align*}
Therefore, we have:
\begin{equation}
P\left(\max_{|z|=1} |P(z)| \leq c\right) \geq \dfrac{1}{c}.
\end{equation}

\parindent=0mm \vspace{.1in}
{\bf Remark:} If $A_0$ and $A_n$ are correlated, then the above lower bound may not be tight. Additionally, the above argument does not give any upper bound on $P\left(\displaystyle\max_{|z|=1} |P(z)| \geq c\right)$.
 
\subsection{Maximum Modulus Bounds}
\begin{theorem}
 Let $\mathcal{P}_n$ denotes the ensemble of all polynomials of degree $n$ with  standard random coefficients, and if $d\mu_P$  denote the uniform  probability measures on $\mathcal{P}_n.$Further,let $ P(z)=\displaystyle\sum_{i=0}^n A_i z^i$  be a random polynomial with  coefficients distributed as standard normal variates and $Y=\displaystyle\max_{|z|=1}P(z).$ Then we have the following :
   \begin{equation}
  \mathbb P\left(\displaystyle\max_{|z|=1} |P(z)| \leq c\right) \geq \dfrac{2}{c}\,\sqrt{\dfrac{2}{ \pi}}.
  \end{equation}   
\end{theorem}
 {\bf Proof:} For the random polynomial $P(z)=\displaystyle\sum_{i=0}^n A_i z^i ,A_i \sim N(0,1),$ denoting $\displaystyle\max_{|z|=1} |P(z)| $ and $ X=|A_0| + |A_n|$, we have $\displaystyle\max_{|z|=1} |P(z)| \geq |A_0| + |A_n|.$  We can compute the expected value of $X$ using lemma 3. However, we will do it directly. 
To that end , we first prove that for a folded normal variate $X$ corresponding to a standard  normal variable ,the expected value has the value:

\begin{equation}
 \mathbb E (X)=\sqrt{\dfrac{ 2}{\pi}}.
 \end{equation}

The expected value of a random variable is defined as the integral of the variable times its probability density function $(pdf).$ For a standard normal random variable $X$, the $pdf$ is given by $f(x) = \frac{1}{\sqrt{2\pi}} e^{-\frac{x^2}{2}}$.

The absolute value of $ X $ is a non-negative random variable, so we can split the expectation into two integrals over the positive and negative reals:

\begin{equation}
 \mathbb E[|X|] = \int_{-\infty}^{0} -x f(x) dx + \int_{0}^{\infty} x f(x) dx
\end{equation}

Because the standard normal distribution is symmetric, these two integrals are equal, so we can write:

\begin{equation}
 \mathbb E[|X|] = 2 \int_{0}^{\infty} x f(x) dx
\end{equation}

Substituting the $pdf$ into the integral, we get:

\begin{equation}
\mathbb E[|X|] = 2 \int_{0}^{\infty} x \frac{1}{\sqrt{2\pi}} e^{-x^2/2} dx
\end{equation}

This is a standard form of integral, and its solution is $\sqrt{2/\pi}$, so we have:

\begin{equation}
\mathbb E[|X|] = \sqrt{\dfrac{2}{\pi}}.
\end{equation}

This establishes the claim.

Hence,
  $ |A_0| ,\, |A_n|$  being folded normal variables corresponding to standard normal variables have both mean equal to $ \sqrt{\dfrac{2}{\pi}}.$ Therefore, in view of independence the expected value of $X$ is  $2\sqrt{\dfrac{2}{\pi}}$. Now, employing Markov inequality, we have 
 \begin{align*}
P\left(\max_{|z|=1} |P(z)|  \leq c\right) & \geq  \dfrac{\mathbb E\left( \max_{|z|=1} |P(z)|\right) }{c}  \\
&  \geq \dfrac{2}{c}\,\sqrt{\dfrac{2}{ \pi}}. 
\end{align*}

\section{Conclusion}
Our exploration into the average modulus of a polynomial over a circle or disc and the average modulus of a random polynomial for prescribed distributions on the coefficients has yielded significant insights into a hitherto understudied area of polynomial theory. The notion average  of a polynomial over a unit disc as a tool to characterize the behavior of a polynomial remains an under-investigated area. Our results seek to fill this research gap. Our choice of coefficients distributed as standard normal variates, a common choice in statistical modeling, has allowed us to explore a wide range of scenarios and natural phenomena, expanding our understanding of the interplay between random systems and polynomials. The resultant probabilistic model and findings have the potential to add depth to our understanding of polynomial behaviors and their applicability across numerous disciplines. Moreover, the presented result concerning the bounds of the maximum modulus has provided us with a new lens through which we can examine polynomial behavior. Utilizing the principles of probability, we derived results concerning the likelihood of the maximum modulus exceeding a specific threshold, with the Markov inequality serving as a foundational tool in our analysis.
\section{Future Work}
 The groundwork laid by this study presents several promising avenues for future research that can be pursued by the broader academic community. One potential direction is the examination of random polynomials with coefficients following non-Gaussian distributions, which could offer a richer understanding of polynomials in environments with different types of noise or uncertainty. Furthermore, the extension of our results to random polynomial systems with dependencies among coefficients would be invaluable, particularly for modeling complex networks and systems where interactions between components cannot be ignored. Another interesting avenue would be to investigate the implications of the results discussed in the paper in higher-dimensional spaces and for polynomials of multiple variables, which are of particular interest in multivariate statistical models and high-dimensional data analysis. Lastly, considering the vast applications of polynomials in various scientific and engineering disciplines, we encourage interdisciplinary research to adapt and test our theoretical findings in practical, real-world scenarios, potentially leading to novel insights and advancements in those fields.

\section*{Declarations}
{\bf Conflict of Interest:} The authors declare that there is no conflict of in interests.

\parindent=0mm \vspace{.1in}
{\bf Funding:}Not Applicable


\parindent=0mm \vspace{.2in}
\begin{thebibliography}{10}

\bibitem{rah}
Rahman, Q. I., and Schmeisser, G.  ``Analytic Theory of Polynomials: Critical Points, Zeros and Extremal Properties." Oxford University Press, 2002.

\bibitem{borp}
Borwein, P., and Erdelyi, T. ``Polynomials and polynomial inequalities." Springer,  1995.

\bibitem{Bharuch} Bharucha-Reid, A. T. and Sambandham, M. ``Random polynomials." Academic Press, Orlando, 1986.


\bibitem{mil} Milovanovic, G.V., Mitrinovic, D.S., Rassias, T., ``Topics in polynomials: extremal properties, inequalities, zeros." World Scientific Publishing Co., Singapore, 1994.

\bibitem{sheikh2023}
S.A. Sheikh, M.I. Mir, J.G. Dar, I.M. Almanjahie, and F. Alshahrani, 
``A Probabilistic Version of Eneström–Kakeya Theorem for Certain Random Polynomials,'' 
\textit{Mathematics}, vol. 11, no. 4061, 2023. 
\url{https://doi.org/10.3390/math11194061}


\bibitem{sheikh2024}
S.A. Sheikh, M.I. Mir, O.A. Alamri, J.G. Dar,
"On Variance and Average Moduli of Zeros and Critical Points of Polynomials,"
\textit{Symmetry}, vol. 16, no. 349, 2024.
\url{https://doi.org/10.3390/sym16030349}

\bibitem{neder} 
Visser, C. 
A Simple Proof of Certain Inequalities Concerning Polynomials. 
\emph{Nederl. Akad. Wetensch., Proc.}, \emph{48}, 276–281, \textbf{1945}. 

\bibitem{deewan} Chanam, B., Dewan,  K.K., Inequalities for a polynomial and its derivative, Journal of Mathematical Analysis and Applications, 336(1), (2007).



\bibitem{grigoriev2002} 
P. G. Grigoriev, 
"Estimates for Norms of Random Polynomials," 
\textit{East Journal on Approximation}, vol. 7, no. 4, pp. 445-469, 2001. 
Available at: \url{http://arxiv.org/abs/math/0210342v1}.
\bibitem{borwein2001}

P. Borwein and R. Lockhart, 
"The Expected $L_p$ Norm of Random Polynomials," 
\textit{Proceedings of the American Mathematical Society}, vol. 129, no. 5, pp. 1463-1472, May 2001. 
Available at: \url{http://www.jstor.org/stable/2668757}.
\bibitem{salzyg}
R. Salem and A. Zygmund, Some properties of trigonometric series whose terms have random signs, Acta Math. 91(1954) 245–301.

\bibitem{oppen99} Oppenheim, A. V., Schafer, R. W., \& Buck, J. R. (1999). Discrete-time signal processing (Vol. 2). Upper Saddle River, NJ: Prentice Hall.

\bibitem{kuo92} Kuo, B. C. (1992). Automatic control systems (Vol. 1). Englewood Cliffs, NJ: Prentice Hall.

\bibitem{prok} Proakis, J. G., \& Manolakis, D. G. (2001). Digital signal processing: principles, algorithms, and applications. Prentice Hall.

\bibitem{franklin} Franklin, G. F., Powell, J. D., \& Emami-Naeini, A. (1994). Feedback control of dynamic systems (Vol. 3). Reading, MA: Addison-Wesley.

\bibitem{filter} Lutovac, M. D., Tošić, D. V., \& Evans, B. L. (2000). Filter design for signal processing using MATLAB and Mathematica. Prentice Hall.
\bibitem{rebon} Rebonato, R. (2004). Volatility and correlation: The perfect hedgers and the foxes. John Wiley \& Sons.
\bibitem{glasser} Glasserman, P. (2004). Monte Carlo methods in financial engineering (Vol. 53). Springer Science \& Business Media.

\bibitem{land80} Landau, D. P., \& Binder, K. (1980). A guide to Monte Carlo simulations in statistical physics. Cambridge University Press.







\bibitem{frap} Frappier, C., Rahman, Q. I. and Rusciteweyh, S., New Inequalities for polynomials, Trans. Amer. Math. Soc. 28(1)  69-99, (1958).

\bibitem{govil} Govil, N. K. Inequalities for the derivative of a polynomial, J. Approx. Theory, 66(1), (1991), 29-35.

\bibitem{Shah} Ahanger, U.M., Shah, W.M. Inequalities for the derivative of a polynomial with restricted zeros. J Anal., 29, 1367–1374 (2021). 



\bibitem{Aziz} Aziz, A., Dawood,  Q.M, Inequalities for a polynomial and its derivative
J. Approx. Theory, 53, 155-162, (1988).

\bibitem{III} Mir, M.I., Nazir, I., Wani, I.A., On Erdös–Lax and Turán-type inequalities for polynomials, Asian-European Journal of Mathematics, 16(3), (2023).

\bibitem{maxbound}
 Dubickas, A. A Lower Bound for the Maximum of a Polynomial in the Unit Disc, Anal Math., 46, 67–76 (2020). 

\bibitem{mathst} Maximum value of a complex polynomial on the unit disk, Mathematics Stack Exchange, (2017). \href{https://math.stackexchange.com/q/2278087}{https://math.stackexchange.com/q/2278087}

\bibitem{jain-iyengar}
R. K. Jain and S. R. K. Iyengar,
\textit{Advanced Engineering Mathematics},
5th edition,
Narosa Book Distributors, 2016,
ISBN: 9788184875607.
\bibitem{Ross} Ross, S. M. ``Introduction to Probability Models." Academic Press, 2014.

\bibitem{Feller} Feller, W. ``An Introduction to Probability Theory and Its Applications." John Wiley \& Sons, 2008.
\end{thebibliography}
\end{document}